\documentclass[12pt,reqno]{amsart}

\usepackage{amsmath,amssymb,amsthm,bm,mathrsfs,braket,sasaki,color}
\usepackage{graphicx}
\makeatletter
\@addtoreset{equation}{section}
\makeatother

\theoremstyle{plain}
\newtheorem{theorem}{Theorem}[section]

\newtheorem{lemma}[theorem]{Lemma}

\newtheorem{corollary}[theorem]{Corollary}

\theoremstyle{definition}

\newcommand{\LR}{L^2(\BR)}

\theoremstyle{remark}

\newcommand{\mmm}[4]
{\left(
\!\!\!\begin{array}{cc}#1&#2\\ #3&#4\end{array}\!\!\!\right)
}

\newcommand{\vvv}[1]
{\left(
\!\!\!\begin{array}{l}#1\end{array}\!\!\!\right)
}
\newcommand{\EEE}{E_0}
\newcommand{\half}{\frac{1}{2}}
\newcommand{\han}{{1/2}}
\newcommand{\kak}[1]{(\ref{#1})}
\newcommand{\ff}{g_1}%\Gamma_1}
\newcommand{\fff}{g_2}%\Gamma_2}
\newcommand{\gr}{\varphi_{\rm g}}
\renewcommand{\d}{\displaystyle}
\newcommand{\ms}{\mathscr }
\newcommand{\PPP}{{\rm P}}
\newcommand{\QQQ}{{\rm R}}

\newcommand{\bi}{\begin{description}}
\newcommand{\ei}{\end{description} }
\newcommand{\ab}{\varepsilon}%(\beta-\alpha)}%\delta_{\alpha\beta}}

\newcommand{\oo}[1]
{\langle #1\rangle_{11}}

\newcommand{\ot}[1]
{\langle #1\rangle_{12}}

\renewcommand{\tt}[1]
{\langle #1\rangle_{22}}

\title{  Multiplicity of the lowest eigenvalue of  non-commutative harmonic oscillators  }
\author{Fumio Hiroshima}
\address{Faculty of Mathematics, Kyushu University,  Fukuoka, 819-0395, Japan}
\email{hiroshima@math.kyushu-u.ac.jp.}
\author{Itaru Sasaki}
\address{ Fiber-Nanotech Young Researcher Empowerment Center, 
 Shinshu University, Matsumoto 390--8621, Japan}
\email{isasaki@shinshu-u.ac.jp}

\date{\today}

\keywords{non-commutative harmonic oscillator, multiplicity, lowest eigenvalue}
\subjclass[2000]{35P05, 35P15}

\begin{document}
%%%%%%%%%%%%%%%%
\maketitle

\begin{abstract}
The multiplicity of the lowest eigenvalue $E$ of the so-called non-commutative harmonic oscillator $Q(\alpha,\beta)$ is studied. 
 It is shown that $E$ is 
 %at most two fold degenerate for $\alpha>2$ and $\beta>2$, and is 
 simple for $\alpha$ and $\beta$ in some region. 
 \end{abstract}

%%%%%%%%%%%%%%%%%%%%%%%%%%%%%%%%%%%%%%%%%%%%

%\Markboth{semi-relativistic Nelson}{ground state}

\setlength{\baselineskip}{15pt}

\section{Definition and main results}
Recently a special attention is payed to studying the  spectrum of self-adjoint operators with {\it non-commutative} coefficients.
It is considered not only as mathematics but also physics experiments. 
A historically important model is the Dirac operator, and the 
Rabi model and the Jaynes-Cumming model are prevalent in cavity QED.  
See \cite{hh12} and references therein. 
The non-commutative harmonic oscillator is a quantum system defined by the Hamiltonian:
\begin{align}
Q = Q (\alpha,\beta) = A\tensor \left(-\frac{1}{2}\frac{d^2}{dx^2}+\frac{1}{2}x^2 \right)
      + J\tensor \left(x\frac{d}{dx}+\frac{1}{2}\right),
\end{align}
where $A= \begin{pmatrix} \alpha & 0\\ 0&\beta\end{pmatrix}$,
$J= \begin{pmatrix}0 & -1 \\ 1&0\end{pmatrix}$, and 
$\alpha,\beta>0$ parameters 
with 
 $\alpha\beta>1$.
Operator  $Q $ defines  a positive self-adjoint operator acting in the Hilbert space $\cH={\mathbb C}^2\otimes\LR$. 
%  $\alpha$ and $\beta$ are real parameter assumed to be 
%$\{(\alpha,\beta)\in [0,\infty)\times[0, \infty)|\alpha\beta>1\}$.

The non-commutative harmonic oscillator $Q$ has been introduced by Parmeggiani and Wakayama \cite{pw01,pw02a,pw02b,pw03}, 
and  the  spectral property of $Q $ is considered in \cite{p04, p06, p08a} from the pseudo-differential-calculus point of view.  
It can be seen that 
 $Q $ has purely discrete spectrum 
$\lambda_1\leq \lambda_2\leq \cdots\leq \lambda_n\leq\cdots \uparrow \infty$, 
where the eigenvalues are counted  with multiplicity. 
One can define the so-called spectral zeta function associated with $Q$ as 
 $$\zeta_Q(s)=\sum_{n=1}^\infty \frac {1}{\lambda_n^s}.$$ 
 When $\alpha=\beta$, $Q$ is unitarily equivalent to the direct some of harmonic oscillators, and 
 $\lambda_{2m-1}=\lambda_{2m}=\sqrt{\alpha^2-1}(m +\half)$,  
 and thus $\zeta_Q$ with $\alpha\not=\beta$  can be regarded as 
a $q$-deform of the Riemann zeta-function. 
Analytic properties of the spectral zeta-function is studied in \cite{iw05a, iw05b, iw07, kw06,kw07,ky09}. 
Furthermore 
it is also known that the set of {\it odd} eigenvectors of  non-commutative harmonic oscillator  is deeply related to the 
set of some solutions of the Heun differential  equation \cite{iw05b,o01,o04}:  
$$\frac{\partial^2}{\partial w^2} f+
\left(
 \frac{1-n}{w}+\frac{-n}{w-1}+\frac{n+3/2}{w-a}\right)
 \frac{\partial}{\partial w} f+
\frac{-(3/2)nw-q}{w(w-1)(w-a)}f=0,$$
where $n\in{\mathbb N}\cup\{0\}$, $a\in{\mathbb C}$ with $|a|<1$ and $q\in{\mathbb C}$.

In this paper we concentrate on the study of the lowest eigenvalue
 $\lambda_1$ of $Q$.
We set 
\begin{equation}
E = \lambda_1
\end{equation}
and $p=-id/dx$.
 In particular we are interested in determining the dimension of ${\rm Ker}(Q-E)$. 
The eigenvector associated with   the lowest eigenvalue is called the  ground state.
In the case of $\alpha=\beta$, as is mentioned above,
 $Q $ can be diagonalized as
 $Q \cong \mmm h 0 0 h$ with
\begin{align}\label{cho}
h=
\half p^2+\frac{\alpha^2-1}{2}x^2,
\end{align}
where $\cong$ denotes the unitary equivalence.
Then all the eigenvalues of $Q (\alpha,\alpha)$ are two-fold degenerate.
In particular, its lowest eigenvalue 
\begin{align}
\label{lowest}
\EEE  = \half \sqrt{\alpha^2-1}
\end{align}
is two-fold degenerate.
In the general case, $\alpha\not=\beta$,  the so-called 
Ichinose-Wakayama bound  is established in \cite{iw07}:
\begin{align}
\left (j-\frac{1}{2}\right) \min\{\alpha,\beta\}  \sqrt{\frac{\alpha\beta-1}{\alpha\beta}}
 \leq \lambda_{2j-1}  \leq \lambda_{2j} \leq 
 \left(j-\frac{1}{2}\right)
 \max\{ \alpha, \beta\}
 \sqrt{\frac{\alpha\beta-1}{\alpha\beta}} \label{iwbounds}
\end{align}
for $j\in {\mathbb N}$. 
By this  inequality
we see that
 the multiplicity of $E$ is 
at most  two   if  $\beta <3\alpha$ or $\alpha<3\beta$.

Furthermore beyond above results one may expect that $E$ is simple for $\alpha\not=\beta$. 
In \cite{nnw02} it can be numerically shown that $E$ is simple for $\alpha\not=\beta$ and,
 in \cite{p04} the  simplicity is proven but only for  sufficiently large $\alpha\beta$. 
It is then mentioned in \cite[8.3 Notes]{p08b} that 
the determining the multiplicity of the lowest eigenvalue should be explored.

In this paper we show that 
\bi
\item[(a)]  $E$ is at most twofold degenerate  for 
$(\alpha,\beta)\in (2,\infty)\times(2,\infty)$, 
\item[(b)]
 $E$ is simple for some region of $\alpha$ and $\beta$. 
\ei
In order to prove (a), we apply the method in \cite{hir05}, where 
the so-called  pull-through formula \cite{gj68}  is  a key. 
The second result (b) consists of two estimates. %can be devided into two parts.
The first is for  large $|\beta-\alpha|$ and 
the second  for small $|\beta-\alpha|$ but $\alpha\neq \beta$.
The first case is proven in a similar manner to (a) 
 and the second  by the regular perturbation theory of discrete spectrum.%eigenvalues.

 Let ${\ms G}={\rm Ker}(Q -E)$ be the set of ground states.
Let $L_+\subset \LR$ (resp. $L_-$)  be the set of even functions (resp.odd functions).
We define $\cH_\pm={\mathbb C}^2\otimes L_\pm$. 
Since $Q $ conserves the parity, $Q $ is reduced by $\cH_\pm$. 
Set $Q \lceil_{\cH_\pm}=Q _\pm$ and then $Q =Q _+\oplus Q _-$. 

%%%%%%%%%%%%%%%% theorem 1
\begin{theorem}{\label{th1}}
Suppose that 
$(\alpha,\beta)\in (2,\infty)\times(2,\infty)$.
Then  ${\rm dim} {\ms G}\leq 2$ 
 and ${\ms G}\subset \cH_+$. 
I.e., 
the multiplicity of $E$  is at most two and ground states are  even functions.
\end{theorem}
%%%%%%%%%%%%%%%% end thm
We can furthermore  show that $E$ is simple. 
%, we obtain the following results:
%%%%%%%%%%%%%%%% theorem2
 \begin{theorem}{\label{th2}}
Suppose $\beta>\alpha>2$ and
\begin{align}
 \frac{1}{2} > \left(\frac{1}{2}\beta -E \right)^{-2}\frac{1}{\alpha^2-1}+ \frac{1}{\alpha^2-1}. \label{th2cond}
\end{align}
Then $E$  is simple.
 \end{theorem}
%%%%%%%%%%%%%%%% end theorem
Condition \kak{th2cond} includes the implicit value $E$.
Let 
\begin{equation*}
E_\mathrm{upper}
 =\frac{\sqrt{\alpha\beta}\sqrt{\alpha\beta-1}}{\alpha+\beta+|\alpha-\beta|
   \frac{(\alpha\beta-1)^{1/4}}{\sqrt{\alpha\beta}}{\rm Re} \rho},
\end{equation*}
where $\rho^2= \sqrt{\alpha\beta-1}-i$ with $\mathrm{Re}\, \ome>0$, i.e., 
$\d {\rm Re} \, \rho=\sqrt{\frac{\sqrt{\alpha\beta}(\sqrt{\alpha\beta-1}+1)}{2}}$.
Bound $E<E_\mathrm{upper}$ holds. See \cite[Theorem 8.2.1]{p08b}.
Combining this  with  \kak{th2cond}
we have the corollary:
\begin{corollary}\label{colo13}
Suppose $\beta>\alpha>2$ and
\begin{align}
\label{co2cond} \frac{1}{2} > \left(\frac{1}{2}\beta -E_\mathrm{upper}
  \right)^{-2}\frac{1}{\alpha^2-1}+ \frac{1}{\alpha^2-1}.
\end{align}
Then $E$  is simple.
 \end{corollary}
%The figure \ref{fig1} is the region of $\alpha,\beta$ such that Corollary \ref{colo13} holds,
%but we also drawed the case of $\alpha>\beta$.
%%%%%%%%%%%%%%% fig1
\begin{figure}[t]\label{fig100}
\begin{center}
\includegraphics[width=8cm]{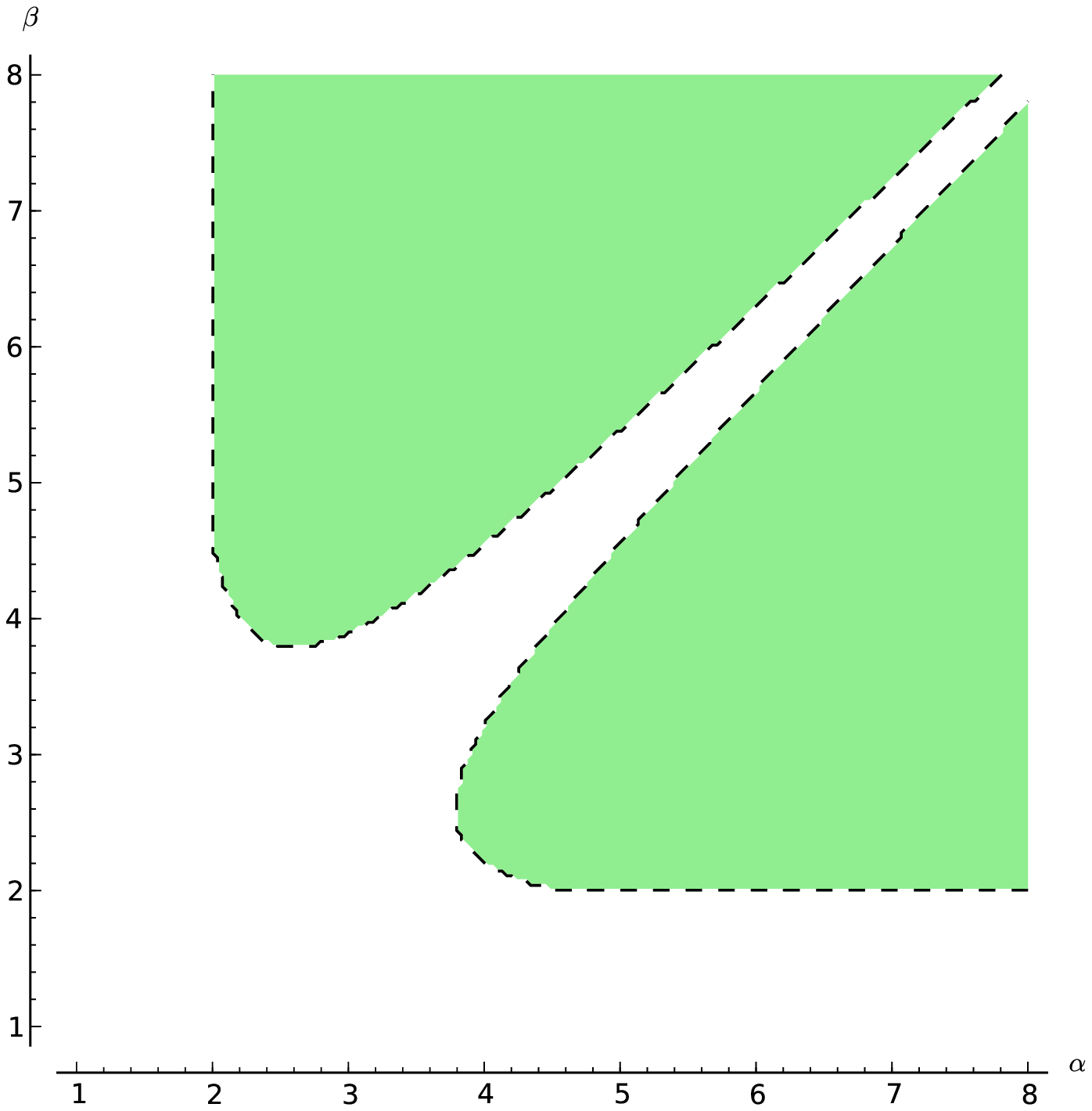}
 \caption{The region satisfying \kak{co2cond}}
(The case of $\alpha>\beta$ is also drawn.)
\end{center}
 \end{figure}
%%%%%%%%%%%%%%%% 
Theorem \ref{th2} does not valid for $(\alpha,\beta)$ 
in  a neighborhood of 
the diagonal
 line on $\alpha-\beta$ plane. See Figure  \ref{fig100}. 
We can also however show that $E$ is simple  for $\alpha$ and $\beta$ in a neighborhood of  the diagonal line.
  We define $g_1,...,g_4$  by 
  \begin{align}
& \label{g1}
\ff = (\alpha-1)^{-1} \left( 3 + \frac{\sqrt{3}}{ \sqrt{\alpha^2-1} } \right),\\
& \label{g2}\fff= \frac{\sqrt{\alpha^2-1}}{2\left| \sqrt{\alpha^2-1} - \lambda_2 \right|}
\ff^2 \\
&\label{g3} g_3 = \frac{\alpha}{2\sqrt{\alpha^2-1}} \\
& \label{g4}g_4 = \frac{(\sqrt{\alpha^2-1})^{3/2}}{4\alpha^{3/2}}.%& \label{l}
%\ell(\ab ) = (1-\ab^2 g_2) \sqrt{1 - 2 \ab^2 g_2}\\
\end{align}
%In order to cover this region, we give the next theorem:
%%%%%%%%%%%%%%%%
 \begin{theorem}{\label{th3}}
Let $\ab =\beta-\alpha$.
 Assume that $\beta>\alpha>1$,
  $\sqrt{\beta^2-1}\leq 3\sqrt{\alpha^2-1}$
and $\ab ^2 \fff<1/2$.
Let
\begin{align}
&\kappa(\ab ) = \EEE  g_1^2 + 
\ab  g_2(\EEE g_1+g_3+g_4) + 
\ab ^22 \EEE  g_1^2 g_2 + \ab ^3 2 \EEE  g_1  g_2^2,\\
&
\ell(\ab)=(1-\ab^2 g_2) \sqrt{1-2\ab^2 g_2^2}.
\end{align}
Then 
\begin{align}\label{positive}
|\lambda_1-\lambda_2|
\geq
\frac{2 \ab}{\ell(\ab)} (g_4 - \ab  \kappa(\ab ))
\end{align}
In particular 
when 
$
 \ab  \kappa(\ab ) < g_4$, 
 $E$ is simple. 
 \end{theorem}
Note that we know the bound 
$\lambda_2\leq\frac{\beta}{2}\frac{\sqrt{\alpha\beta-1}}{\sqrt{\alpha\beta}}$ by the Ichinose-Wakayama bound.
Then, the region of $\alpha,\beta$ satisfying  $\ab  \kappa(\ab ) < g_4$, 
includes 
 a  wedge-shaped  region illustrated in Figure \ref{fig200},
  where we also drew the case of $\alpha>\beta$.
%%%%%%%%%%%%%%%% figure 2
\begin{figure}[t]\label{fig200}
\begin{center}
\includegraphics[width=8cm]{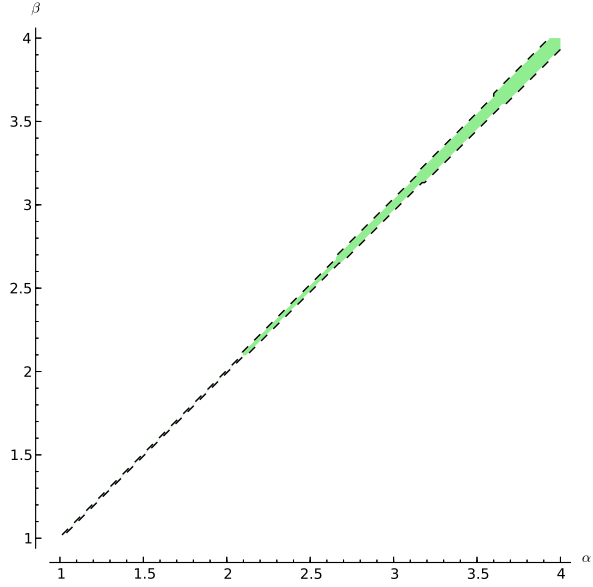}
 \caption{A wedge-shaped region included in the region satisfying 
  $\ab  \kappa(\ab ) < g_4$}
  \end{center}
\end{figure}
%%%%%%%%%%%%%%%%

\section{Proofs of theorems}
\subsection{Proof of Theorem \ref{th1}}
In the following we omit the symbol $\tensor$ for the notational simplicity, and we can suppose that $\alpha<\beta$ without loss of generality.
Let 
$$N=\half \begin{pmatrix} 1&0\\0&1\end{pmatrix}({p^2+x^2-1}) .$$
The spectrum of $N$ is $\sigma(N)=\{0\}\cup {\mathbb N}$ and the multiplicity of each eigenvalues.
Let $\d a=\frac{1}{\sqrt 2}({x+ip})$ and $\d a^*=\frac{1}{\sqrt 2} (x-ip)$. They satisfy canonical commutation
relations $[a,a^*]=1$ and $[a,a]=[a^*,a^*]=0$, and we have $a^*a=N$

{\it Proof of Theorem \ref{th1}}: 
In terms of $a$ and $a^*$, the operator $Q $ can be realized as
\begin{align}
 Q  = a^*Aa + \frac{1}{2}A + \frac{1}{2}\left(aJa - a^*Ja\right).
\end{align}
Let $\gr \in {\ms G}$. 
We have $(Q -E)a\gr  = [Q ,a]\gr $.
Since $[Q ,a]= -Aa + Ja^*$ by canonical commutation relation, we have $(Q -E)a\gr  = (-Aa+Ja^*)\gr $
and then $(Q -E+A)a\gr  = Ja^*\gr $. Notice that $Q -E+A\geq \alpha>0$. We have
\begin{align}
 a\gr  = (Q -E+A)^{-1}Ja^*\gr. 
\end{align}
 Taking the norm on both sides above,
we have
\begin{align}
  \norm{a\gr }^2  \leq \frac{1}{\alpha^2} \norm{a^*\gr }^2.
\end{align}
Since $\norm{a\gr }^2 = \inner{\gr }{N\gr }$ and
 $\norm{a^*\Omega}=\inner{\gr }{N\gr }+ \norm{\gr }^2$, we see that
 \begin{align}
  \inner{\gr }{N\gr } \leq \frac{1}{\alpha^2-1}\norm{\gr }^2. \label{ptb1}
 \end{align}
Let $\PPP_\Omega$ be the projection onto  $\ker N = \ker a$.
Note that $N+\PPP_\Omega\geq 1$.
Let ${\ms G}={\ms G}_+\oplus {\ms G}_-$, where ${\ms G}_\pm={\ms G}\cap \cH_\pm$.
Let $\PPP ^\pm$ be the projection onto ${\ms G}_\pm$.
 Then, by \eqref{ptb1}, we have
\begin{align}
 \PPP ^+ \PPP_\Omega \PPP ^+ \geq \PPP ^+ (1-N)\PPP ^+ \geq 
 \left(1-\frac{1}{\alpha^2-1}\right)\PPP ^+.
\end{align}
Taking the trace of both sides, we have
\begin{align}
 2\geq \mathrm{Tr}(\PPP ^+ \PPP_\Omega \PPP ^+)\geq \frac{\alpha^2-2}{\alpha^2-1}\mathrm{Tr} \PPP ^+
\end{align}
Thus we have the bound
\begin{align}
 \dim \ker \PPP ^+ \leq 2\frac{\alpha^2-1}{\alpha^2-2}.
\end{align}
Then the right hand side above is less than three for $\alpha>2$. Then  ${\rm dim} {\ms G}^+\leq 2$.
Similarly but replacing $\PPP ^+$ with $\PPP ^-$, we can also see that 
\begin{align}
 \PPP ^- \PPP_\Omega \PPP ^- \geq \PPP ^- (1-N)\PPP ^- \geq
 \left(1-\frac{1}{\alpha^2-1}\right)\PPP ^-.
\end{align}
Note that $\PPP ^- \PPP_\Omega \PPP ^-=0$, since $\PPP_\Omega$ is the projection to the set of even functions. Then we have
\begin{align}
 0= \mathrm{Tr}(\PPP ^- \PPP_\Omega \PPP ^-)\geq \frac{\alpha^2-2}{\alpha^2-1}\mathrm{Tr} \PPP ^-.
\end{align}
In particular for $\alpha>\sqrt 2$, the dimension of ${\ms G}_-$ equal to zero. 
Then the theorem follows. 
\qed

\subsection{Proof of Theorem \ref{th2}}
In this subsection we show that the lowest eigenvalue is simple. 
The strategy is parallel with 
that of previous subsection but $\PPP^+ $ is replaced by a projection $\QQQ $ with the dimension of ${\rm Ran} \QQQ =1$. 
Let $\sigma_1,\sigma_2,\sigma_3$ be the $2\times 2$ Pauli matrices given by 
\begin{align}
 \sigma_1 = \begin{pmatrix} 0 & 1 \\ 1 & 0 \end{pmatrix},\quad
 \sigma_2 = \begin{pmatrix} 0 & -i \\ i & 0 \end{pmatrix},\quad
 \sigma_3 = \begin{pmatrix} 1 & 0 \\ 0 & -1 \end{pmatrix}.
\end{align}
{\it Proof of Theorem \ref{th2}}: 
 The Hamiltonian $Q $ can be written in the form:
\begin{align}
  Q  = \half A ({p^2+x^2}) + \half \sigma_2 ({px+xp}).
\end{align}
We set $M=\frac{1}{2}(1+\sigma_3)=\begin{pmatrix} 1 & 0 \\ 0 & 0 \end{pmatrix}$,
$M^\bot = \frac{1}{2}(1-\sigma_3)
=\begin{pmatrix} 0 & 0 \\ 0 & 1 \end{pmatrix}
$ and $\QQQ =M^\bot \PPP_\Omega$.
Then we have
\begin{align}
(Q -E) \QQQ  \gr   = M^\bot [Q ,\PPP_\Omega] \gr  + [Q ,M^\bot]\PPP_\Omega \gr . 
\end{align}
The commutator $[Q ,M^\bot]$ can be computed as 
\begin{align}
   [Q ,M^\bot]    
   %&= -[Q , \frac{1}{2}(1-\sigma_3)]
      %                   = \frac{1}{2} [ Q , \sigma_3 ]
    %= \frac{1}{2}\cdot \frac{i}{2}(a^2 - a^\ast a^\ast ) [\sigma_2,\sigma_3] \notag  \\
    %&   = \frac{i}{4}2i\sigma_1 (aa-a^\ast a^\ast)
    = -\frac{1}{2} (aa-a^\ast a^\ast)\sigma_1.
\end{align}
Thus we have
\begin{align}
 [Q ,\QQQ ] =  M^\bot \frac{i}{2}\sigma_2[aa-a^\ast a^\ast,\PPP_\Omega]
 +\frac{1}{2}a^\ast a^\ast \sigma_1 \PPP_\Omega,
\end{align}
where we used the fact that $a\PPP_\Omega =0$.
Hence we have
\begin{align}
\inner{\QQQ \gr }{(Q -E)\QQQ  \gr }  
& = \inner{\gr }{\frac{i}{2}\QQQ M^\bot \sigma_2 [aa-a^\ast a^\ast,\PPP_\Omega]\gr }
  + \frac{1}{2} \inner{\gr }{\QQQ a^\ast a^\ast\sigma_1\PPP_\Omega\gr } \notag \\
%&= \frac{i}{2} \inner{\gr }{M^\bot \sigma_2 \PPP_\Omega \Big((aa-a^\ast a^\ast)\PPP_\Omega-\PPP_\Omega (aa-a^\ast a^\ast)\Big)\gr }
 % + \frac{1}{2} \inner{\gr }{M^\bot \PPP_\Omega a^\ast a^\ast \sigma_1\PPP_\Omega\gr } \notag \\
&= \frac{i}{2} \inner{\gr }{\QQQ \sigma_2 a^2\gr }.
\end{align}
On the other hand
$
  \QQQ (Q -E)\QQQ  
  %= M^\bot \PPP_\Omega(Q  - E) \PPP_\Omega M^\bot
                               = \left(\frac{1}{2}\beta - E \right)\QQQ ^2$.
Then \eqref{ptb1} and $\|a^\ast \QQQ \gr\|=\|\QQQ \gr\|$ yield that 
\begin{align*}
 \left(\frac{1}{2}\beta - E\right) \norm{\QQQ \gr }^2
 \leq \frac{1}{2}\norm{a^*\QQQ \gr } 
 \norm{\sigma_2a\gr }
 \leq \frac{1}{\sqrt{\alpha^2-1}}\norm{\QQQ \gr }  \norm{\gr }. \label{po}
\end{align*}
Therefore
\begin{align*}
 \norm{\QQQ \gr }^2 \leq \left(\frac{1}{2}\beta - E \right)^{-2}\frac{1}{\alpha^2-1}\norm{\gr }^2.
\end{align*}
Since $M+M^\bot =1$, it holds that
$
  \PPP_\Omega M + \QQQ  + N \geq 1$.
 Then, by using \eqref{ptb1} we have 
\begin{align*}
 \PPP  (\PPP_\Omega M ) \PPP 
\geq \PPP  (1-\PPP_\Omega M^\bot -N) \PPP
%& \geq \PPP  - (\frac{1}{2}\beta - E)^{-2}\frac{1}{\alpha^2-1}\PPP  - \frac{1}{\alpha^2-1} \PPP  \notag \\
 = \left(1- \left(\frac{1}{2}\beta -E\right)^{-2}\frac{1}{\alpha^2-1}- \frac{1}{\alpha^2-1 } \right) \PPP 
\end{align*}
where $P=P^++P^-$ is the orthogonal projection onto ${\ms G}$.
Taking the trace of both sides above, 
 we have
\begin{align*}
 1 \geq \left(1- \left(\frac{1}{2}\beta -E\right)^{-2}\frac{1}{\alpha^2-1} - \frac{1}{\alpha^2-1 } \right) {\rm Tr}  \PPP ,
\end{align*}
and the theorem follows.
\qed

\subsection{Proof of Theorem \ref{th3}}
Recall that $\ab=\beta-\alpha$. 
In this section, we fix an arbitrary $\alpha>1$ and set 
\begin{align}
  Q (\alpha,\beta)=Q=Q(\ab ) = Q _0  + \ab  V,
\end{align}
where
$Q _0 = Q (\alpha,\alpha)$ and 
$ V =\half  \begin{pmatrix}  0&0\\ 0&   1 \end{pmatrix} ({p^2+x^2}) $.
 \begin{lemma}{\label{relbound}}
For all $\Phi \in D(Q _0)$, it follows that
\begin{align}
 \norm{V\Phi} \leq (\alpha-1)^{-1} \norm{Q _0 \Phi} + \frac{\sqrt{3}}{2}(\alpha-1)^{-1}\norm{\Phi}.
 \label{relb}
\end{align}
 \end{lemma}
%%%%%%%%%%%%%%%% end lemma
 \begin{proof}
One can show that $\norm{(px+xp)u}^2 \leq \norm{(p^2+x^2)u}^2 + 3\norm{u}^2$.
Since $\sigma_2^2=1$, we have
\begin{align*}
 \norm{Q _0\Phi} 
\geq  & \alpha \norm{\half ({p^2+x^2})\Phi} - \norm{\half ({px+xp})\Phi} \notag \\
 \geq & \alpha \norm{\half ({p^2+x^2})\Phi}
            - \frac{1}{2} \left( \norm{(p^2+x^2)\Phi} ^2 +3\norm{\Phi}^2\right)^{1/2},
\end{align*}
and hence
\begin{align} 
\norm{Q _0\Phi}\geq&  \frac{\alpha-1}{2}\norm{(p^2+x^2)\Phi} 
                                    - \frac{\sqrt{3}}{2} \norm{\Phi}.
\end{align}
Noticing  $\norm{V\Phi} \leq  \frac{1}{2}\norm{(p^2+x^2)\Phi}$ 
we have the bound \eqref{relb}.
\end{proof}
%%%%%%%%%%%%%%%%
By \kak{iwbounds} or the sandwich estimate  $Q (\alpha,\alpha)\leq Q (\alpha,\beta)\leq Q (\beta,\beta)$
we see bounds:
$
 \frac{\sqrt{\alpha^2-1}}{2} \leq \lambda_1 \leq \lambda_2 \leq  \frac{\sqrt{\beta^2-1}}{2} $
and
$
  \frac{3}{2}\sqrt{\alpha^2-1} \leq \lambda_3$.
When  $\sqrt{\beta^2-1} \leq 3\sqrt{\alpha^2-1}$, 
 $Q$ has 
 exactly two eigenvalues in 
 the  
 interval $\left[
 \frac{1}{2}\sqrt{\alpha^2-1} , \frac{1}{2} \sqrt{\beta^2-1} 
 \right] $.
Let $C$ be the closed  disk  centered at $\frac{\sqrt{\alpha^2-1}}{2}$ with 
the radius $\frac{\sqrt{\alpha^2-1}}{2}$ in the complex plane:
\begin{align}
 C = \left\{\left. \frac{\sqrt{\alpha^2-1}}{2}+r e^{i\theta}\in {\mathbb C} \right|
 0\leq r\leq \frac{\sqrt{\alpha^2-1}}{2}, 
 0\leq \theta\leq 2\pi\right\}.
\end{align}
Thus $Q$ has exactly  two eigenvalues, $\lambda_1$ and $\lambda_2$, inside of  $C$.
 
%%%%%%%%%%%%%%%% lemma
 \begin{lemma}{\label{relb2}}
It follows that $  \norm{V(Q _0-z)^{-1}}  \leq \ff$ 
 for all $z\in \del C$. 
 \end{lemma}
 \begin{proof}
 By Lemma \ref{relbound}, we have
\begin{align}
 \norm{V(Q _0-z)^{-1}} \leq (\alpha-1)^{-1}\norm{Q _0(Q _0-z)^{-1}} + \frac{\sqrt{3}}{2}(\alpha-1)^{-1}
 \norm{(Q _0-z)^{-1}}.
\end{align}
Since the eigenvalues of $Q _0$ are 
$\{(\frac{1}{2}+n)\sqrt{\alpha^2-1}\}_{n=0}^\infty$
we have 
 $\sup_{z\in \partial C} \norm{Q _0(Q _0-z)^{-1}} = 3$ 
 and 
 $
  \sup_{z\in \partial C}\norm{(Q _0-z)^{-1}} = \frac{2}{\sqrt{\alpha^2-1}}$.
Then the lemma follows. 
\end{proof}
%%%%%%%%%%%%%%%%
The two-dimensional subspace spanned by eigenvectors  associated with eigenvalues
$\lambda_1$ and $\lambda_2$ is denoted by $\ms F $.
The orthogonal projection onto 
$\ms F $ 
is then given by 
\begin{align}
 P=P(\ab) = -\frac{1}{2\pi i} \oint_{\partial C} (Q-z)^{-1} dz.
\end{align}
 We expand $P(\ab  )$ with respect to 
  $\ab $ up to   the second order:
\begin{align}
\label{secondorder}
 P = P_0 +\ab  P_1+ \ab ^2 R,
\end{align}
where $P_0$ is the orthogonal projection onto the ground states of $Q _0$ and
\begin{align}
 P_1 &= -\frac{1}{2\pi i} \oint_{\partial C} (Q _0 -z)^{-1} V (Q _0-z)^{-1}dz, \\
R= R(\ab ) &=
 -\frac{1}{2\pi i} \oint_{\partial C} (Q _0 -z)^{-1} V (Q _0-z)^{-1} V (Q-z)^{-1}dz.
\end{align}
%In order to estimate the above operators, we prepare the following lemma:
%%%%%%%%%%%%%%%% lemma
 \begin{lemma}{\label{boundr}}
We have
$
 \norm{R}
   \leq
   \fff$ and 
   $\norm{VP_1} \leq \EEE \ff^2$.  
    \end{lemma}
%%%%%%%%%%%%%%%% lemma
 \begin{proof}
By Lemma \ref{relb2}, we have
\begin{align}
 \norm{R}
\leq
 \frac{|C|}{2\pi} \sup_{z\in \partial C} \norm{V(Q _0-z)^{-1}}^2 \norm{(Q-z)^{-1}} 
  \leq \frac{\sqrt{\alpha^2-1}}{2} \ff^2  \norm{(Q-\sqrt{\alpha^2-1})^{-1}}.
\end{align}
Since $\lambda_3 \geq \frac{3}{2}\sqrt{\alpha^2-1}$, 
we have $ \norm{(Q-\sqrt{\alpha^2-1})^{-1}} = |\lambda_2-\sqrt{\alpha^2-1}|^{-1}$.
Hence 
$\norm{R}
   \leq
   \fff$ holds. Similarly one can prove the second bound. 
   \end{proof}
%%%%%%%%%%%%%%%%
Let $v_0 \in L^2(\BR)$ be the normalized ground state of
$
h=\half p^2+\frac{{\alpha^2-1}}{2}x^2$.
Namely
\begin{align}
\label{harmonic}
 v_0(x) = \left( \frac{\sqrt{\alpha^2-1}}{\pi}\right)^{1/4} e^{-\sqrt{\alpha^2-1} x^2/2}.
 \end{align}
 
Let $S_a$ be the dilation defined by $S_af(x)=\frac{1}{\sqrt a}f(ax)$ for $a>0$. 
We define the unitary operator  $U$ on $\cH$ by
\begin{align}
\label{U}
U= \frac{1}{\sqrt{2}} S_{\sqrt\alpha} 
\begin{pmatrix}   e^{ix^2/(2\alpha)} & 0 \\ 0 & e^{-ix^2/(2\alpha)}  \end{pmatrix}
\begin{pmatrix} 1& -i \\ 1 & i
\end{pmatrix}.
\end{align}
Then
$
U Q _0 U^\ast
  =
\begin{pmatrix} h & 0\\
0 & h
  \end{pmatrix}
$
and vectors
$u_1=U^\ast  \begin{pmatrix} v_0 \\ 0 \end{pmatrix}$ and
$u_2=U^\ast\begin{pmatrix} 0 \\ v_0 \end{pmatrix}
$
are two fold ground states of $Q (\alpha,\alpha)$.
{ Since $P$ is a projection onto $\ms F $,
each of the ground state and the first excited state can be expresses as  a linear combination 
of $Pu_1$ and $Pu_2$
as long as both $Pu_1$ and $Pu_2$ are linearly independent}, which is 
proven in the lemma below:
%%%%%%%%%%%%%%%%
\begin{lemma}{\label{rboun}}
Assume that $\ab ^2 \fff<1$. Then 
$Pu_1$ and $Pu_2$ are non-zero vectors.
Moreover, if $\ab^2 \fff<1/2$, then $Pu_1$ and $Pu_2$ are linearly independent.
\end{lemma}
\begin{proof}
By \kak{secondorder} we have 
\begin{align}
 \norm{Pu_1}^2 = 1 + \ab  \inner{u_1}{P_1u_1} + \ab ^2\inner{u_1}{R u_1}.
\end{align}
The second term on the right hand side above is zero, since 
\begin{align}
 \inner{u_1}{P_1u_1} 
 %&= \frac{-1}{2\pi i} \oint_{\partial C} \inner{u_1}{(Q _0-z)^{-1}V (Q _0-z)^{-1}u_1}dz 
 %\notag \\
=  
 \frac{-1}{2\pi i} \oint_{\partial C} \frac{1}{(E_0 -z)^{2}}dz  \inner{u_1}{Vu_1}  =0.
\end{align}
%where $\EEE =\sqrt{\alpha^2-1}/2$ is the lowest eigenvalue  of $Q_0$.
Hence $\norm{Pu_1}^2 \geq 1- \ab ^2\fff^2>0$ 
holds by Lemma \ref{boundr} and the assumption.
Thus $P u_i (i=1,2)$ are non-zero vectors.
Next we assume that $\ab ^2 g_2 < 1/2$. Then we have
\begin{align}
 \norm{Pu_1}^2  \norm{Pu_2}^2 - |\inner{Pu_1}{Pu_2}|^2
 &\geq 
 (1-\ab ^2 g_2)(1 - \ab ^2 g_2) - \ab ^4 g_2^2 \notag \\
 &= 1 - 2 \ab ^2 g_2 > 0
\end{align}
which implies that $Pu_1$ and $Pu_2$ are linearly independent.
 \end{proof}
%%%%%%%%%%%%%%%%
\begin{lemma} \label{uvu}
Let $g_3$ and $g_4$ be given in \kak{g3} and \kak{g4}, respectively.  
Then it  follows that 
where we used the bound 
\begin{align}
& \label{b1}\norm{Vu_i} \leq  \ff\EEE,\quad i=1,2,\\
& \label{b2}|\inner{u_1}{Vu_1}| = g_3,\\% := \frac{\alpha}{2} (\alpha^2-1)^{-1/2} \\
& \label{b3}|\inner{u_1}{Vu_2}|   \leq  g_4.
\end{align}
\end{lemma}
\proof
\kak{b1} follows from 
\begin{align}
 \norm{Vu_i} \leq  \norm{V(Q _0-\sqrt{\alpha^2-1})^{-1}}   \norm{(Q _0-\sqrt{\alpha^2-1})u_i}
\leq \ff\EEE.
\end{align}
The proofs of \kak{b2} and \kak{b3} 
are given in Appendix. 
\qed

%%%%%%%%%%%%%%%%%%%%%%%%%%%%%%%%%%%%%%
{\it Proof of Theorem \ref{th3}}: 
Suppose that 
%In the following, we assume that
$  \ab ^2 \fff <1/2$.
We define
\begin{align}
 \Phi_1 = \frac{Pu_1}{\norm{Pu_1}}, \qquad
 \Phi_2 = \frac{Pu_2 - \inner{\Phi_1}{u_2}\Phi_1}{\norm{Pu_2 - \inner{\Phi_1}{u_2}\Phi_1}}
\end{align}
Then, by Lemma \ref{rboun}, $\Phi_1$ and $\Phi_2$ are orthogonal vectors in $\ms F $. 
Let ${\ms V}=L.H.\{\Phi_1,\Phi_2\}$ be the two dimensional vector space and 
$Q:\ms V\to \ms V$ can be regarded 
as a linear operator and its matrix representation 
is given by 
\begin{align}
 m= \begin{pmatrix}
  \inner{\Phi_1}{Q\Phi_1} & \inner{\Phi_1}{Q\Phi_2} \\
  \inner{\Phi_2}{Q\Phi_1} & \inner{\Phi_2}{Q\Phi_2} \\
 \end{pmatrix}.
\end{align}
 Thus the eigenvalues $\lambda_1$ and  $\lambda_2$ of $Q$ 
 are also the eigenvalue of $m$. 
Therefore, the difference of $\lambda_1$ and $\lambda_2$
can be computed by
\begin{align}
( \lambda_2-\lambda_1)^2
 = \left(\inner{\Phi_1}{Q\Phi_1} - \inner{\Phi_2}{Q\Phi_2}\right)^2
    + 4 |\inner{\Phi_1}{Q\Phi_2}|^2,
\end{align}
which implies that 
\begin{align}
| \lambda_2-\lambda_1|
\geq 
    2 |\inner{\Phi_1}{Q\Phi_2}|.
\end{align}
We estimate $|\inner{\Phi_1}{Q\Phi_2}|$ from below.
Inserting the definition of $\Phi_j$  into $ \inner{\Phi_i}{ Q\Phi_j}$ we have 
\begin{align} 
 \inner{\Phi_1}{Q\Phi_2} = 
 \frac{ \inner{Pu_1}{Q( \norm{Pu_1}^2 Pu_2-(Pu_1, Pu_2)Pu_1)}}
{\|Pu_1\|^3 \|Pu_2-(Pu_1, Pu_2)Pu_1/\|Pu_1\|^2\|}. \label{h1}
\end{align}
Notice that 
\begin{align*}
 \inner{Pu_i}{QPu_j} = (Pu_i,Qu_j)
  = \EEE  \delta_{ij}+\ab\langle V\rangle_{ij} + \ab^2  (\EEE  \langle R\rangle_{ij}
    +\langle P_1V\rangle_{ij}) + \ab^3\langle RV\rangle_{ij},
\end{align*}
where $\langle K \rangle_{ij}=(u_i, K u_j)$.
We see that the denominator of \kak{h1} is expanded as 
\begin{align}
&\|Pu_1\| \|\|Pu_1\|^2 Pu_2-(Pu_1, Pu_2)Pu_1\|   \notag \\
&=
(1+\ab^2 \oo R) \sqrt{(1+\ab^2 \oo R)(1+\ab^2 \tt R)-\ab^4 |\ot{R}|^2} \notag. 
%& \geq  (1-\ab^2 g_2) \sqrt{(1-\ab g_2)(1 - \ab g_2) -\ab ^4 g_2^2} \notag,
%& =: \ell(\ab ), \notag \\
\end{align}
By the bound $\norm R\leq g_2$ we have the lower bound 
\begin{align}
\|Pu_1\| \|\|Pu_1\|^2 Pu_2-(Pu_1, Pu_2)Pu_1\|   
 \geq  (1-\ab^2 g_2) \sqrt{1-2\ab^2 g_2^2}.
\end{align}
The  numerator of \kak{h1} can be also expanded as 
\begin{align*}
&{(Pu_1, Q( \norm{Pu_1}^2 Pu_2-(Pu_1, Pu_2)Pu_1))}\\
&=\ab \ot{V}
+\ab^2 \ot{P_1 V}
+\ab^3 \Big(\ot{RV}+\ot{V}\oo{R}-\ot{R}\oo{V} \Big) \\
&\quad 
+\ab^4 \Big( \ot{P_1 V} \oo{R}  -\ot{R}\oo{P_1V} \Big)
+\ab^5 \Big( \oo{R} \ot{RV}  -   \ot{R}\ot{RV} \Big) .
\end{align*}
By using Lemmas \ref{boundr} and \ref{uvu}, each term can be evaluated as 
\begin{align*}
& \ab^2 | \ot{P_1 V}|\leq \ab^2  \EEE  g_1^2 \\
& \ab^3 | \ot{RV}+\ot{V}\oo{R}-\ot{R}\oo{V}| 
   \leq   \ab^3 \left(  \EEE  g_1 g_2  +g_2g_4 +g_2g_3 \right) \\
&\ab^4| \ot{P_1 V} \oo{R} - \ot{R}\oo{P_1V}   | 
    \leq \ab^4 2\EEE  g_1^2 g_2\\
&\ab^5| \oo{R}\ot{RV}-\ot{R}\ot{RV} |
   \leq \ab^5 2 \EEE   \ff \fff^2.
\end{align*}
By combining all  the estimates stated above, we have
\begin{align}
|\lambda_1-\lambda_2|
\geq
\frac{2\ab }{\ell(\ab )}
\left( g_4 - \ab  \kappa(\ab ) \right), 
\end{align}
where 
$\ell(\ab)=(1-\ab^2 g_2) \sqrt{1-2\ab^2 g_2^2}$.
%(1+\ab^2 \oo R) \sqrt{(1+\ab^2 \oo R)(1+\ab^2 \tt R)-\ab^4 |\ot{R}|^2}$.
Hence  the theorem follows.  
\qed 

\appendix
\section{Computation of $\langle V\rangle_{ij}$}
We recall that 
$v_0=(\omega/ \pi)^{1/4}e^{-\omega x^2/2}$ 
with $\omega=\sqrt{\alpha^2-1}$, 
$V=\half \mmm 0 0 0 1 (p^2+x^2)$ and $U$ is given by 
\kak{U}.
We directly see that 
\begin{align*}
&UVU^{-1}\\
&=\frac{1}{4}
 S_{\sqrt\alpha}
\mmm{e^{ix^2/2\alpha}} 0 0 {e^{-ix^2/2\alpha}}
\mmm 1 {-i} 1 i
 \mmm 0 0 0 1 
 \mmm 1 1 i {-i} 
(p^2+x^2)
\mmm{e^{-ix^2/2\alpha}} 0 0 {e^{ix^2/2\alpha}}
S_{1/\sqrt{\alpha}}
\\
&=
\frac{1}{4}
\mmm
{e^{ix^2/2\alpha}(\frac{p^2}{\alpha}+\alpha x^2)e^{-ix^2/2\alpha}}
{e^{-ix^2/2\alpha}(\frac{p^2}{\alpha}+\alpha x^2)e^{ix^2/2\alpha}}
{-e^{-ix^2/2\alpha}(\frac{p^2}{\alpha}+\alpha x^2)e^{-ix^2/2\alpha}}
{e^{ix^2/2\alpha}(\frac{p^2}{\alpha}+\alpha x^2)e^{ix^2/2\alpha}}.
\end{align*}
Then we have 
\begin{align}
\label{v11}
\inner{\vvv {1\\ 0}}{UVU^{-1}\vvv {1\\0}}_{{\mathbb C}^2}&
=\frac{1}{4}
e^{ix^2/2\alpha}\left( \frac{p^2}{\alpha}+\alpha x^2\right) 
 e^{-ix^2/2\alpha}\\
\label{v12}
\inner{\vvv {1\\ 0}}{ UVU^{-1}\vvv {0\\1}}_{{\mathbb C}^2}
&
=-\frac{1}{4}
e^{ix^2/2\alpha} \left( 
\frac{p^2}{\alpha}+\alpha x^2\right)
 e^{ix^2/2\alpha}.
\end{align}
\begin{lemma}
\label{v1}
It follows that $\d \inner{u_1}{Vu_1}=\frac{\alpha}{4\omega }$, and then 
$\d |\inner{u_1}{Vu_1}|=g_3$.
\end{lemma}
\proof
By \kak{v11} we have 
$$\inner{u_1}{Vu_1}=\frac{1}{4}\inner{v_0}{e^{ix^2/2}\left(\frac{p^2}{\alpha}+\alpha x^2\right) e^{-ix^2/2}v_0}=
\frac{1}{4}\inner{v_0}{\left(\frac{(p-x)^2}{\alpha}+\alpha x^2\right) v_0}.$$
Since $\inner{v_0}{(px+xp)v_0}=0$, 
we obtain that 
$$\inner{u_1}{Vu_1}=
\frac{1}{4}
\inner{v_0}{\left(\frac {p^2+x^2}{\alpha}+\alpha x^2\right) v_0}.$$
By 
$\inner{v_0}{x^2v_0}=1/(2\omega)$ and $\inner{v_0}{p^2v_0}=\omega/2$
 we have 
$\inner{u_1}{Vu_1}=\alpha/(4\omega)$.
\qed

\begin{lemma}
\label{v2}
It follows that 
$\d \inner{u_1}{Vu_2}=-\frac{\omega^{3/2}}{4\alpha}
(\omega-i)^{-1/2}$.
In particular $\d |\inner{u_1}{Vu_2}|\leq \frac{\omega^{3/2}}{4\alpha^{3/2}}$. 
\end{lemma}
\proof
By \kak{v12} we have 
\begin{align*}
\inner{u_1}{Vu_2}
&=-\frac{1}{4}\inner{v_0}{e^{ix^2/2}\left(\frac{p^2}{\alpha}+\alpha x^2\right)e^{ix^2/2}v_0}\\
&=-\frac{1}{4\alpha}\inner{(-i-\omega)xe^{-ix^2/2}v_0}{(i-\omega)xe^{ix^2/2}v_0}-\frac{\alpha}{4}\inner{v_0}{x^2 e^{ix^2}v_0}\\
&=
-\frac{1}{4\alpha}((\omega-i)^2+\alpha^2)\inner{v_0}{x^2 e^{ix^2}v_0}.
\end{align*}
Since $\d \inner{v_0}{x^2 e^{ix^2}v_0}=(\omega/\pi)^{1/2}\int_{\mathbb R} x^2 e^{-(\omega-i)x^2} dx 
=\half \omega^\han  (\omega-i)^{-3/2}$, 
we have the lemma.
\qed

\noindent{\bf Acknowledgment}

We thank Takashi  Ichinose and Masato Wakayama for sending \cite{iw07} to our attention and giving 
helpful comments. 
I. S.'s work was partially supported by  Y22740087 from JSPS,
and was performed through the Program for Dissemination of Tenure-Track System
funded by the Ministry of Education and Science, Japan. 
F.H. thanks for the financial support by 
Grant-in-Aid for Science Research (B) 20340032
from JSPS.

\end{document}